\documentclass{amsproc}

\usepackage[cmtip,all]{xy}

\usepackage{}

\newtheorem{theorem}{Theorem}[section]
\newtheorem{lemma}[theorem]{Lemma}

\theoremstyle{definition}

\theoremstyle{remark}

\numberwithin{equation}{section}
\usepackage[usenames,dvipsnames]{color}

\theoremstyle{plain}
\newtheorem{thm}{Theorem}[section]

\theoremstyle{definition}
\newtheorem{dfn}[thm]{Definition}

\theoremstyle{remark}
\newtheorem{rmk}[thm]{Remark}

\DeclareMathOperator{\Tr}{Tr}

\begin{document}

\title{Five not-so-easy pieces{:}\ open problems about vertex rings}

\author{Geoffrey Mason}
\address{Department of Mathematics, University of California, Santa Cruz}
\curraddr{}
\email{gem@ucsc.edu}
\thanks{We thank the Simons Foundation grant $\#427007$, for its support.}


\date{}

\begin{abstract}
We present five open problems in the theory of vertex rings.\ They cover a variety of different
areas of research where vertex rings have been, or are threatening to be, relevant.\ They have also been chosen because I personally find them interesting,\ 
and because I think each of them has a chance (the title of the paper notwithstanding!) of being solved.\ In each case we give some explanatory background
and motivation, sometimes including proofs of special cases.\ Beyond vertex rings \emph{per se}, the topics covered include connections 
to \emph{real Lie theory, formal group laws, modular linear differential equations, Pierce bundles, and genus 2 Siegel modular forms and the Moonshine Module.}
\end{abstract}

\maketitle

 \section{Introduction}
 The organizers of this Conference kindly asked if I might be interested in writing a paper about some open problems in 
 the theory of vertex operator algebras.\ I liked the sound of the idea but they provided no further guidance, and I struggled with
 the question of an appropriate format and likely topics.\ Here's what I have \emph{not} included: readvertising well-known current problems,
 for example the question of Mathieu Moonshine and its umbral variants \cite{EOT}, \cite{CHD}, \cite{TG}, though this is indeed a fascinating 
 area.\ Similarly, I've skipped the vast question of explaining the parallels between VOA theory and the theory of subfactors.\ I thought it could be 
 worthwhile to discuss the following question:\ why is it so hard to construct VOAs?\ The long struggle to rigorously construct the VOAs on Schellekens list (i.e., the holomorphic,
 $c{=}24$ VOAs), now complete thanks to the efforts of many, is an illustration of this question.\ The putative Haggerup VOA of Terry Gannon would have made a worthy centerpiece
 of such a discussion.\ However, Gannon has recently circulated a preprint on this subject  \cite{TG1} and in doing so he unwittingly vitiated my idea.
 
 \medskip
 In the end I've included questions according to the following criteria{:} (i) I've thought about them and find them interesting; (ii) they collectively exhibit some diversity of topic;
 (iii) I believe they're all doable.\
 The  titles of each Section are as follows{:}
  
 \medskip\noindent
 -\  Real forms of vertex operator algebras.\\
 -\  Vertex rings and formal group laws.\\
 -\ Vertex operator algebras and modular linear differential equations.\\
 -\ Pierce bundles of local vertex rings.\\
 -\ Genus 2 Monstrous Moonshine.

\section{Real forms of vertex operator algebras}
The following is a fundamental fact in Lie theory: let $\frak{g}$ be a \emph{semisimple} complex Lie algebra.\ Then $\frak{g}$ has a
\emph{compact real form}.\ (Compact means that the restriction of the Killing form to a Cartan subalgebra of a real form of $\frak{g}$  is \emph{negative-definite}.)\\
\\
$\mathbf{Problem\ 1}$.\ Describe the class of complex VOAs $V$ which have an analog of this theorem.\ Do strongly regular VOAs always have a real form?

\medskip
Some discussion is appropriate to provide context and meaning to this problem.

\medskip
In the following, by a VOA we always mean a \emph{complex} VOA.\ For the sake of clarity it is sometimes convenient to include a subscript on vertex operators and modes to indicate which Fock space is intended:\ thus vertex operators for $V$ are $Y_V(u, z){=}\sum_n u_V(n)z^{-n-1}$, etc.

\medskip
Let $V$ be a complex linear space.\ We follow what is fairly standard notation and let $V^{\mathbf{R}}$ denote the same set $V$  \emph{regarded as a real linear space}.\ If $V$ is also a Lie algebra then $V^{\mathbf{R}}$ is naturally a real Lie algebra.\ What about VOAs?
If $V$ is a VOA with 
$u{\in}V^{\mathbf{R}}$ we can still define vertex operators
\begin{eqnarray*}
Y_{V^{\mathbf{R}}} (u, z){:=}\sum_n u_V(n)z^{-n-1}.
\end{eqnarray*}
This means that $u_{V^{\mathbf{R}}}(n)$ is the mode $u_V(n)$ \emph{regarded as a real operator on $V^{\mathbf{R}}$}.

\medskip
Let 
$\omega{\in}V$ be the Virasoro element of $V$ whose modes $L(n)$  satisfy the standard identity
\begin{eqnarray*}
[L(m), L(n)]{=}(m{-}n)L(m{+}n){+}\frac{m^3{-}m}{12}\delta_{m, -n}cId_V,
\end{eqnarray*}
where $c$ is the \emph{central charge} of $V$.\ The statement that $\omega$ is a Virasoro element of $V^{\mathbf{R}}$ \emph{makes no sense
unless $c$ is a real number}.\

\medskip
We shall say that $\omega$ is \emph{real} if the central charge $c$ that it determines is also real.\
 (We do not want to say that $V$ is real in this situation; it could be misleading.)

\medskip\noindent
\begin{dfn}\label{defreal} Let $(V, Y, \mathbf{1}, \omega)$ be a (complex) VOA with vertex operator $Y$, vacuum element $\mathbf{1}$, and a real Virasoro element $\omega$.\ A \emph{real} form of $V$ is a \emph{conformal subVOA} $U{\subseteq} V^{\mathbf{R}}$ such that
\begin{eqnarray}\label{defrealform}
V^{\mathbf{R}}{=}U \oplus iU.
\end{eqnarray}
\end{dfn}

A conformal subVOA of $V^{\mathbf{R}}$ is a real subVOA $(U, Y, \mathbf{1}, \omega)$ whose vacuum and Virasoro elements are the same as that of $V$.\ Note that
 multiplication by $i$ defines an isomorphism of $U$-modules $i{:}U\stackrel{\cong}{\longrightarrow} iU$.\ With a decomposition such as (\ref{defrealform}) in hand we can say, without invoking any ambiguity, that the states in $U$ are the \emph{real states} and those in $iU$ are the \emph{imaginary states}.\ For example, $\mathbf{1}$ and $\omega$ are both real states.\
If $a$ and $b$ are real states, and if $u{:=}a{+}ib{\in}V^{\mathbf{R}}$,  then
\begin{eqnarray}\label{mode}
Y_{V^{\mathbf{R}}}(u, z){=} Y_U(a{-}b, z){\oplus}iY_U(a{+}b, z).
\end{eqnarray}

\medskip
 Let us now assume that $V$ is of \emph{strong CFT-type}.\ In particular  $V$ is a \emph{simple} VOA with a conformal decomposition of the shape
 \begin{eqnarray*}
V{=}\mathbf{C}\mathbf{1}{\oplus}V_1{\oplus}\hdots,
\end{eqnarray*}
 furthermore $V_1$ consists of \emph{quasiprimary states}, i.e., if $u{\in}V_1$ then $L(1)u{=}0$.\ With this set-up, Haisheng Li has shown
 \cite{Li1} that  $V$ has, up to scalars, a unique \emph{invariant bilinear form} 
 \begin{eqnarray*}
b{:}V{\times}V\rightarrow\mathbf{C}.
\end{eqnarray*}
For the definition of an invariant bilinear form for a VOA, cf.\  \cite{FHL}, \cite{Li1}. In particular, reference \cite{FHL} establishes that such a form is necessarily \emph{symmetric}.\ Furthermore because $V$ is simple then $b$ is \emph{nondegenerate}, in particular its restriction to each homogeneous space
$V_n$ is also nondegenerate.\ We may, and shall,  canonically \emph{normalize} $b$ so that
\begin{eqnarray*}
b(\mathbf{1}, \mathbf{1}){=}{-}1.
\end{eqnarray*}
\emph{The resulting bilinear form is the best VOA analog of the Killing form for a semisimple Lie algebra.}

\medskip
Next we observe that Li's theory applies equally well to real VOAs as well as complex VOAs.\ In particular, a real VOA of $CFT$-type
has a unique normalized, \emph{real-valued}, invariant, bilinear form.

\medskip
Let us now assume that $V$ is not only of strong $CFT$-type, but in addition it has a real Virasoro element $\omega$ and a real form $U$.\ Because $\omega{\in}U$  it follows from (\ref{mode}) that
\begin{eqnarray*}
L_V(n){=}L_U(n){+}iL_U(n).
\end{eqnarray*}
Thus $V_n{=}U_n{+}iU_n$.\ It follows that $U_0{=}\mathbf{R}\mathbf{1}$ and $U_1$ is annihilated by $L_U(1)$, so that $U$ is a real VOA of $CFT$-type.\
As a result, $U$ also has unique normalized, invariant, real-valued, bilinear form.

\begin{lemma}\label{lemb} Suppose that $V$ is a (complex) VOA that has strong $CFT$-type, a real Virasoro element and a real form $U$.\ Let $b$ be the normalized invariant bilinear form
on $V$.\ Then  the restriction of $b$ to $U{\times}U$ is the normalized real-valued, invariant, bilinear form on $U$.\ 
\end{lemma}
\begin{proof}
The restriction $b_U$ of $b$ to $U$ is an invariant bilinear form on $U$, and we have to show that it is
\emph{real-valued}.

\medskip
The real and imaginary parts of $b$, $\Re(b)$ and $\Im(b)$ both define real-valued, invariant, bilinear forms on $U$, so there is a real scalar
$y$ such that  $y\Re(b_U){=} \Im(b_U)$.\ Hence restriction to $U$ satisfies $b_U{=}\Re(b_U){+}i\Im(b_U){=}(1{+}iy)\Re(b_U)$.\ We now have
\begin{eqnarray*}
{-}1{=}b_U(\mathbf{1}, \mathbf{1}){=}(1{+}iy)({-}1),
\end{eqnarray*}
so $y{=}0$.\ This proves the Lemma.
\end{proof}

Problem 1 may now make sense.\ Given a (complex) VOA satisfying the conditions of Lemma \ref{lemb} we have two related and canonically defined
invariant bilinear forms, namely the $\mathbf{C}$-valued normalized, invariant, bilinear form $b$ on $V$ and its restriction to the $\mathbf{R}$-valued normalized, invariant, bilinear form on $U$.\ This is a VOA analog of the Lie algebra scenario in which we have a semisimple Lie algebra $\frak{g}$ and the restriction of the Killing form to a Cartan subalgebra of a real form for $\frak{g}$.

\medskip
Problem 1 has two parts.\ Unlike the case of semisimple Lie algebras, there is currently no available result guaranteeing that a suitable class of
VOAs (satisfying the hypotheses of Lemma \ref{lemb}, say) actually has a real form.\ One part of Problem 1 asks for some resolution of this problem, and 
specifically asks if VOAs that are \emph{strongly regular} might fit the bill.\ Strongly regular VOAs are the nicest class of VOAs of all \cite{M1}:\ by definition
they are of  strong $CFT$-type and they are also rational and $C_2$-cofinite.\ If $V$ is strongly regular then the central charge is
\emph{rational}  \cite{DLM}, so that the Virasoro element is certainly real.\ However it is generally unknown if $V$ has a real form.\ Many strongly regular VOAs are known to have a $\mathbf{Z}$-form \cite{DG1}, \cite{DG2}, so they certainly have a real form.\ \emph{A strongly regular VOA without a real form would itself be interesting.}\ 

\medskip
The other part of Problem 1 deals with \emph{compactness}, i.e., the (negative-) definiteness of some real bilinear forms.\ It is somewhat vague
in detail, and deliberately so.\ Signatures of real bilinear forms have not played a big part in VOA theory, although the fact that the Griess algebra
$V_2^{\natural}$ of the Moonshine module supports a real and compact bilinear form is important for Norton identities in monstrous moonshine \cite{Mats}.\ Beyond this I am unaware of 
even so much as a passing reference to such questions, save for the comment of Kac and Raina at the end of Lecture 8 in \cite{KR}.

\medskip
One analog of the Lie theory set-up goes as follows: with earlier notation, and assuming $V$ has a real form $U$, can we choose $U$ so that 
the restriction of $b$ to each $U_n$ is compact? This seems like too much to expect in general, but what about lattice theories $V_L$?\ In the Lie theory set-up we only consider restrictions to a Cartan subalgebra, so for VOAs we might hope to find some 
real subVOA $U'{\subseteq}U$, not too small, such that $b$ is compact on each $U'_n$, or at the very least, $b$ has a computable signature on each $U'_n$.

\medskip
Suppose that $V$ is strongly regular.\ Then the Lie algebra on $V_1$ is  \emph{reductive}, and in some sense a Cartan subalgebra of $V_1$
plays the r\^{o}le of a Cartan subalgebra of $V$.\ (For more on this perspective, see \cite{M1}).\ By the Lie theory we can find a Cartan subalgeba of $V_1$ having a compact real form.\
Perhaps this is all that one can hope for?\ On the other hand by Theorem 1 of \cite{M1}, for any Cartan subalgebra $C{\subseteq}V_1$ we can find
a lattice theory subVOA $V_{L}{\cong}W{\subseteq}V$ such that $L{\subseteq}C$ is cocompact.\ This suggests taking $C$ to have a compact real form and
looking for $U'$ in $W$.\ Can we take $U'$ to be a real form of $W$?

\section{Vertex Rings and Formal Group Laws}

\noindent
$\mathbf{Problem\ 2}.$  Let $k$ be a commutative ring, $F$ a formal group law over $k$ and $V$ a vertex $k$-algebra.\ What are the
axioms for an $F$-vertex $k$-algebra?

\medskip
 Haisheng Li made  some substantial (and surprising, to me) contributions towards fostering the connections between VOAs and formal group laws (FGLs) in \cite{Li2}.\
  Li showed how to modify the weak associativity axiom for a VOA using a FGL $F$.\ In this way he produced what he called a \emph{vertex $F$-algebra},
 which is a variant of a VOA that satisfies axioms resembling the locality axioms  (i.e., the Jacobi identity, locality, and weak commutativity and associativity) but are modified by $F$.

\medskip
Problem 2 asks for a generalization of Li's results to \emph{vertex rings} rather than VOAs.\ (In our nomenclature it was convenient to replace vertex $F$-algebra with $F$-vertex algebra, but be assured they mean the same thing.)\ Vertex rings are just like VOAs, but the scalars lie in an  arbitrary commutative ring 
$k$ rather  than a field of characteristic $0$ such as $\mathbf{C}$.\ An axiomatic approach to  vertex rings is given in \cite{M3}.\ Although the axioms for a vertex ring are virtually identical to those for a VOA, the lack of denominators in a vertex ring hampers attempts to prove analogs of results for VOAs, which become more complicated, or plain wrong, or meaningless.\ 

\medskip
So it is with FGLs.\ Working with a FGL $F$ over a commutative ring $k$ which is a $\mathbf{Q}$-algebra is facilitated by the existence of a \emph{logarithm} for $F$, implying that
$F$ is isomorphic to the additive group  law.\ Sure enough,  in his applications to VOAs Li makes heavy use of the existence of a logarithm.\ Thus the main point of
Problem 2 is to  reproduce Li's results without using logarithms! (Well, you may use them, but they should not figure in the final answer!)

\medskip
If one is going to meld FGLs with  vertex algebra theory, it seems most natural to do it for vertex rings without denominators, for that is the natural domain in which the theory of FGLs resides.\ And by working at this level of generality one gains access to what could be  some of the most interesting FGL's for vertex algebra theory.\ We have in mind
Euler's FGL associated to an elliptic integral, the Lazard universal FGL, and the FGLs of generalized cohomology theories.

\medskip
In the rest of this Section I will review some of the relevant background about FGLs designed to make the preceding paragraphs intelligible.\ Then I will answer  Problem 2 in the special case of vertex rings of type
$(k, \underline{D})$, where $k$ is, as before, an arbitrary commutative ring and $\underline{D}$ is an \emph{iterative Hasse-Schmidt derivation} of $k$.\ These are the easiest vertex rings that are \emph{not} VOAs (cf.\ \cite{M3}, Theorem 5.5), and they provide a lifeline to the classical theory of rings with derivation \cite{Matsumura}.\ This special  case may suggest what to expect when attacking Problem 2 in full generality.

\medskip
For somewhat different approaches to formal groups and FGLs, we may refer the reader to the encyclopedic text of Hazewinkel \cite{H}, and a more scheme-theoretic approach 
in a course of  Strickland \cite{S} which I found on the internet.\ 

\medskip
Fix a unital commutative ring $k$.\ A \emph{$1$-dimensional commutative formal group law} (FGL) over $k$ is a power series $F(X, Y){\in}k[[X, Y]]$ satisfying the following properties:
\begin{eqnarray*}
&&(a)\  F(X, 0){=}X, F(0, Y){=}Y,\\
&&(b)\ F(F(X, Y), Z){=}F(X, F(Y, Z)),\\
&&(c)\ F(X, Y){=}F(Y, X).
\end{eqnarray*}
These look like the right- and left-identity, associativity and commutativity axioms for an abelian group\ They should be supplemented by an axiom for inverses, but that's not necessary 
because  it is readily proved (formal implicit function theorem) that
\begin{eqnarray*}
\mbox{there is a power series $\iota(X){\in}k[[X]]$ such that}\ F(X, \iota(X)){=}0.
\end{eqnarray*}
We have
\begin{eqnarray*}
F(X, Y){=}X{+}Y{+}\sum_{i, j\geq 0} c_{ij}X^iY^j\ \ (c_{ij}{\in}k).
\end{eqnarray*}
For `nice' rings $k$ it is true that (c) is a \emph{consequence} of (a) and (b) (cf. \cite{H}, Theorem 6.1 for a precise statement).\

\medskip
There are the following basic examples of FGLs, defined for every $k$:
\begin{eqnarray*}
&&\ F_a(X, Y){:=}X{+}Y,\\
&&F_m(X, Y){:=}X{+}Y{+}XY.
\end{eqnarray*}
These are the \emph{additive} and \emph{multiplicative} FGLs respectively.

\medskip 
Let us introduce the notation 
\begin{eqnarray*}
X{+}_F\ Y{=} F(X,Y).
\end{eqnarray*}
It suggests how to produce a \emph{new abelian group structure} on $k[[X]]$ as follows:
\begin{eqnarray*}
a(X){+}_F\ b(X){:=} F(a(X),b(X)).
\end{eqnarray*}

\medskip
\emph{This idea goes to the heart of how and why FGLs may be used to modify a vertex ring.}

\medskip
A \emph{Hasse-Schmidt derivation} (HS)  $\underline{D}{:=}(D_0, D_1, D_2, ...)$ of $k$  is a sequence of endomorphisms $D_m{:}k\rightarrow k$
satisfying 
\begin{eqnarray*}
&&D_0{=}Id_k,\ \ D_m(1){=}0\ (m{\geq}1),\\
&&D_m(uv){=}\sum_{i{+}j=m} D_i(u)D_j(v)\ \ (m\geq0, u, v{\in} k)
\end{eqnarray*}
An HS derivation $\underline{D}$ is called \emph{iterative} if it also satisfies the identity
\begin{eqnarray}\label{iterative}
D_i\circ D_j{=}{i{+}j\choose i}D_{i+j}.
\end{eqnarray}

There is a very interesting generalization of the idea of an iterative HS derivation that uses a FGL, cf. \cite{Matsumura}, \cite{HK}.\ This goes as follows{:}\
\ Let $F(X, Y)$ be a FGL over $k$ and let $\underline{D}$  be an HS derivation of $k$.\ We call $\underline{D}$ an \emph{HS $F$-derivation} if it satisfies the
following identity{:}

\begin{eqnarray}\label{HSF}
\sum_{i, j\geq0} D_j\circ D_i(u)X^iY^j{=}\sum_{n\ge0} D_n(u)(X{+}_FY)^n\ \ (u{\in}k).
\end{eqnarray}

It can be checked \cite{Matsumura} that an HS derivation $\underline{D}$ is iterative if, and only if, it is an $F_a$-derivation (additive FGL).\ So this really \emph{is} a generalization of iterativity.\
This concept has been well-studied in the
literature.\
For  further details in the case when $F$ is the multiplicative FGL $F_m$, cf.\  \cite{CR}, where one gets a glimpse of the rather complicated identities satisfied by the operators $D_m$ that must replace (\ref{iterative}).
 
 \begin{rmk} It may not be quite clear why one needs an FGL in (\ref{HSF}) as opposed to some other power series.\ This point is discussed in \cite{HK}.
 \end{rmk}
 
\medskip
Now let us turn to the vertex rings $(k, \underline{D})$ where $\underline{D}$ is an HS-derivation of $k$.\ They are defined as follows (\cite{M3}, Section 5.2){:}\ the vacuum element is $1{\in}k$ and
vertex operators are defined by
\begin{eqnarray}\label{vertops}
Y(u, z)v{:=}\sum_{n\geq 0} D_{n}(u)vz^n\ \ (u, v{\in}k).
\end{eqnarray}
The relevant result is
\begin{thm}\label{thm3.1}(\cite{M3}, Theorems 3.5{+}5.5) $(k, \underline{D})$ is a vertex ring if, and only if, $\underline{D}$ is \emph{iterative}. $\hfill\Box$
\end{thm}
\noindent

\medskip
In the spirit of Li's work and Problem 3, we now ask:
\begin{eqnarray*}
\emph{\mbox{what happens if, in the above construction, we replace $\underline{D}$ by a HS $F$-derivation?}}
\end{eqnarray*}

We will get what can only be described as a vertex ring modified by an HS $F$-derivation.\ So we call this an \emph{$F$-vertex ring}.\ It is certainly \emph{not}
a vertex ring in the usual  sense unless $F{=}F_a$ is the additive FGL, because only in this case will $\underline{D}$ be iterative.\ So it must be that
one of the locality axioms is changed, and indeed this is the case.\ We have the following generalization of Theorem \ref{thm3.1}.

\begin{thm}\label{thmVF} Let $(k, \underline{D}, F)$ consist of a unital commutative ring $k$, a HS derivation $\underline{D}$ of $k$, and a FGL $F$ over $k$.\
Let vertex operators be defined by (\ref{vertops}).\ Then the following are equivalent{:}
\begin{eqnarray*}
&&(i)\ \ \underline{D}\ \mbox{is an HS $F$-derivation} \\
&&(ii)\ Y(Y(a, z)b), w)c{=}Y(a, z{+}_F\ w)Y(b, w)c\ \ (a, b, c{\in}k)
\end{eqnarray*}
\end{thm}
\begin{proof} (ii)$\Rightarrow$(i).\ Take $b{=}c{=}1$ in (ii) to obtain
\begin{eqnarray*}
&&Y(Y(a, z)1, w)1{=}Y(a, z{+}_F\ w)Y(1, w)1\\
\Rightarrow&&Y\left(\sum_{m\geq 0} D_m(a)z^m, w\right){=}\sum_{m\geq 0}D_m(a)(z{+}_F\ w)^m\\
\Rightarrow&&\sum_{m\geq 0}\sum_{\ell\geq0} D_{\ell}(D_m(a)) w^{\ell}z^m{=}\sum_{m\geq0} D_m(a)(z{+}_F\ w)^m.
\end{eqnarray*}
This is (\ref{HSF}), so (i) holds.

\medskip\noindent
$(i)\Rightarrow(ii)$.\ The left-hand-side of (ii) is equal to
\begin{eqnarray*}
&&\sum_{n\geq0} D_n\left(\sum_{m\geq0} D_m(a)bz^m \right)cw^n\\
{=}&&\sum_{i, j\geq 0} \sum_m D_i\circ D_m(a)D_j(b)cz^m w^{i+j}\ \ \ (\mbox{using the HS property})\\
{=}&&\sum_{k\geq0} D_k(a)(z{+}_F\ w)^k\sum_{j\geq0} D_j(b)cw^j\ \ \ (\mbox{using the $F$-derivation property}) \\
{=}&& Y(a, z{+}_F\ w)Y(b, w)c,
\end{eqnarray*}
and this is the right-hand-side of (ii).\ 
The Theorem is proved.\ 
\end{proof}

\medskip
The thrust of Problem 2 should now be evident.\ It is asking for the generalization of Theorem \ref{thmVF} to \emph{arbitrary} vertex $k$-algebras $V$.\ How do we modify 
$V$ using a FGL $F$, and how are the vertex ring axioms affected?

\medskip
The general case is a good deal more complicated than the special case that we just handled.\  Although
HS derivations figure prominently in all vertex rings (cf.\ \cite{M3}, esp.\ Sections 3-5) it is by no means true in general that $\underline{D}$ determines the vertex operators 
as it does for $(k, \underline{D})$.\ Furthermore in the general case $\underline{D}$ intervenes in the \emph{translation-covariance} axiom (loc.\ cit.), something we did not have to consider in the special case.\ What seems clear is that the HS-derivation
will have to be an HS $F$-derivation just as before.\ But the $F$-weak associativity axiom (this is Li's name for (ii) in Theorem \ref{thmVF}) will almost certainly get generalized to read
\begin{eqnarray}\label{Fweak}
(z{+}_F\ w)^N Y(Y(a, z)b), w)c{=}(z{+}_F\ w)^NY(a, z{+}_F\ w)Y(b, w)c
\end{eqnarray}
($N\gg0$).

\medskip
To recapitulate, we do not want to suggest that Theorem \ref{thmVF} remains true if we simply replace (ii) by (\ref{Fweak}).\ We are suggesting that (\ref{Fweak}) together with
HS $F$-derivations will be key ingredients in the solution of Problem 2.\ One thing is sure{:}\ a close reading of \cite{Li2} will be required!

\section{Vertex operator algebras and modular linear differential equations}\label{S4}

\noindent
$\mathbf{Problem\ 4}.$   State and prove the $3$-dimensional Mathur-Mukhi-Sen theorem.

\medskip 
If I had to name a single paper that has most influenced my thinking about vertex operator algebras, then  the answer would undoubtedly be
Zhu's work on modular-invariance \cite{Z}, which was essentially his Phd\ thesis.\ Zhu deals with a class of nice VOAs that are simple, of CFT-type, rational, and $C_2$-cofinite.\ Thus $V$ is nearly strongly regular, but without any assumptions about existence of an invariant bilinear form, which is not relevant to Zhu's analysis.\ We will simply paraphrase these conditions by saying (somewhat inaccurately) that $V$ is rational.\ Thus $V$ has only finitely many irreducible modules.\ Let's name them $M^1, \hdots, M^p$ and let the conformal weight of $M^i$ be $h^i$.\ If $V$ has central charge $c$ then the \emph{formal graded dimension}, or \emph{$q$-character},  of $M^i$ is defined to be
\begin{eqnarray}\label{q-exp}
Z_i(q){:=}\Tr_{M^i} q^{L(0)-\tfrac{c}{24}}{=}q^{h^i-\tfrac{c}{24}}\sum_{n\geq 0} \dim M^i_n,
\end{eqnarray}
where at first we are obliged to treat $q$ as a formal variable.

\medskip
Zhu's paper is concerned with these $q$-characters.\ He pioneered the use of \emph{differential equations} in
VOA theory by proving that the $q$-characters are \emph{convergent}.\ The method is  well worth studying.\ One uses the Virasoro operators  to show that $Z_i(q)$ satisfies a differential equation which has a \emph{regular singularity} at $q{=}0$.\ The \emph{Frobenius method} then proves convergence in a deleted neighborhood of $q{=}0$.\ (There are many references for background on the theory of differential equations that we need here, ranging from the old-fashioned
\cite{I} and the similar but easier-to-read \cite{Hi}, to the more modern \cite{PS}.)\ 

\medskip
From the perspective of a dyed-in-the-wool algebraist, the beauty of this approach is that it is purely formal, and can be used to great effect in VOA theory.\ All of the estimates and analysis are taken care of by Frobenius!\ At the same time, Zhu proves that if we set $q{:=}e^{2\pi i\tau}$ with $\tau$ in the complex upper half-plane $\mathbf{H}$, then $Z_i(q)$ is just the $q$-expansion
of a periodic holomorphic function $Z_i(\tau)$ in $\mathbf{H}$.\ Thus each $Z_i(\tau)$ has  a certain translation-invariance property.\ Zhu's main theorem is a related and more trenchant invariance result.\ To describe it, introduce the complex linear space 
\begin{eqnarray*}
\frak{ch}_V{:=}\langle Z_1(\tau), \hdots, Z_p(\tau)\rangle\subseteq F.
\end{eqnarray*}
Here, we consider $\frak{ch}_V$ as a subspace of the space $F$ of all holomorphic functions in $\mathbf{H}$.\ Zhu proved that $\frak{ch}_V$ furnishes a representation of the inhomogeneous modular group  $\Gamma{:=}SL_2(\mathbf{Z})$.\ This means that if $\gamma{\in}\Gamma$ then there are scalars $c_{ij}(\gamma){\in}\mathbf{C}$ (independent of $\tau$) such that
\begin{eqnarray*}
Z_i(\gamma\tau){=}\sum_j c_{ij}(\gamma)Z_j(\tau),
\end{eqnarray*}
where we are using standard notation for the fractional linear action $\Gamma{\times}\mathbf{H}{\rightarrow}\mathbf{H}$ given by
\begin{eqnarray*}
(\gamma, \tau){:=}\left(\left(\begin{array}{cc}a & b\\ c & d\end{array}\right), \tau\right)\mapsto \gamma\tau{:=}\frac{a\tau+b}{c\tau+d}.
\end{eqnarray*}

\medskip
For example, if $V$ is \emph{holomorphic} in the sense that it has a \emph{unique} irreducible module ($p{=}1$ and $M^1{=}V$) then Zhu's theorem says that
$Z_1(\tau)$ satisfies
\begin{eqnarray*}
Z_1(\gamma\tau){=}c_{11}(\gamma)Z_1(\tau)\ \ (\gamma{\in}\Gamma)
\end{eqnarray*}
The assignment $\chi{:} \gamma\mapsto c_{11}(\gamma)$ is the representation (a \emph{character} of $\Gamma$ in this case) furnished by $\frak{ch}_V$.\ Thus $Z_1(\tau)$ is a \emph{modular function of weight $0$ and level $1$} with character $\chi$.\ The theory of modular
functions provides a complete description of such functions.\ For example, if $\chi{=}\mathbf{1}$ is the \emph{trivial character} then the
modular functions in question constitute the field $\mathbf{C}(j)$ of rational functions in the absolute modular invariant
\begin{eqnarray*}
j(\tau){=} q^{-1}{+}744{+}196884q{+}\hdots
\end{eqnarray*}

The character $\chi$ does not provide any real difficulty.\ If it is nontrivial then it has order  $3$, and the corresponding class of modular functions is readily described.\ We will not need it here.

\medskip
Attempting to usefully organize all holomorphic VOAs $V$ into some type of classification scheme seems problematic.\ There are infinitely many self-dual, positive-definite, integral lattices $L$, all of which realize holomorphic lattice theories $V_L$, moreover the class of such VOAs is closed with respect to tensor products.\ The $q$-character of $V$ is probably 
as good an  invariant as we can expect, although it does \emph{not} always distinguish between VOAs.\ For example, the two self-dual, rank 16 lattices
  $L_1{:=}E_8{+}E_8$ and $L_2{:=}\Gamma_{16}$ (a spin lattice with root system $D_{16}$) have the same theta-function and therefore
  they define holomorphic lattice theories $V_{L_1}, V_{L_2}$ that are \emph{not} isomorphic but have  \emph{the same} $q$-characters.
  
  \medskip
  This brings us to the Problem stated at the beginning of this Section which concerns a case when $p{>}1$, so that $V$ has more than 
  one\footnote{Of course, this means more than one isomorphism class of irreducible modules} irreducible module.\ Unlike the holomorphic case, there is some hope of classifying such VOAs based on the theory of \emph{modular linear differential equations} (MLDEs).\ For an introduction to this subject (but \emph{not}  its applications in VOA theory) cf.\ \cite{FM};\ we will need some details here.
  
  \medskip
  Fix an integer $k$.\ There is a right action of $\Gamma$  on $F$ defined by
  \begin{eqnarray*}
(f(\tau), \gamma)\mapsto f|_k\gamma (\tau){:=}(c\tau{+}d)^{-k} f(\gamma\tau).
\end{eqnarray*}
Here, and below, we take $\gamma{=}\left(\begin{array}{cc}a & b \\ c & d\end{array}\right)$.

\medskip
Let $F_k$ denote this $\Gamma$-module.\ The \emph{modular derivative in weight $k$} is the differential operator
\begin{eqnarray*}
D_k{:}F{\rightarrow}F, \ \ f\mapsto D_kf{:=} \frac{1}{2\pi i}\frac{df}{d\tau}{-}\frac{k}{12}E_2(\tau)f,
\end{eqnarray*}
where $E_2(\tau)$ is the weight 2 Eisenstein series
\begin{eqnarray*}
E_2(\tau){:=} 1{-}24\sum_{n\geq 1} \sum_{d{\mid}n} dq^n.
\end{eqnarray*}
The  strange-looking operator $D_k$ is important because
\begin{eqnarray*}
D_k{:}F_k\rightarrow F_{k+2}
\end{eqnarray*}
\emph{intertwines} the $\Gamma$-modules $F_k$ and $F_{k+2}$.\ This amounts to the standard formula \cite{L}, Chapter X{:}
\begin{eqnarray}\label{intertwine}
(D_kf)|_{k+2}\gamma(\tau){=} D_k(f|_k\gamma)(\tau).
\end{eqnarray}
 By iteration we obtain
differential operators
\begin{eqnarray*}
D_k^n{:=} D_{k+2n-2}\circ\cdots\circ D_{k+2}\circ D_k{:}F_k\rightarrow F_{k+2n}.
\end{eqnarray*}

\medskip
Let $\widetilde{F}_k$ be the subspace of $F_k$ consisting of holomorphic functions that are `meromorphic at $q{=}0$'.\ 
$\widetilde{F}_k$ is not a $\Gamma$-submodule, however we have
\begin{eqnarray*}
D_k{:}\widetilde{F}_k\rightarrow \widetilde{F}_{k+2}.
\end{eqnarray*}
For example, $E{\subseteq}F_0$ is a $\Gamma$-submodule by Zhu's theorem and $E{\subseteq}\widetilde{F}_0$ because the condition of
meromorphy holds thanks to the $q$-expansions (\ref{q-exp}).\ Hence $D_k^n(E)$ is a $\Gamma$-module contained in $\widetilde{F}_{2n}$

\medskip
The graded algebra of classical holomorphic modular forms is
\begin{eqnarray*}
M{:=}\oplus_{k\geq 0} M_{2k}
\end{eqnarray*}
where $M_{2k}$ is the subspace of $\widetilde{F}_{2k}^{\Gamma}$ (the $\Gamma$-invariants with a $q$-expansion) that are `holomorphic at $\infty$'.\ 
These are the weight $2k$ modular forms.\ Now we are prepared to define an MLDE of  \emph{order $n$ and weight $(2k, 0)$}.\ This is a DE of the type 
\begin{eqnarray}\label{DE}
\left(\sum_{j=0}^nP_{2k+2n-2j}(\tau)D_0^j\right)u{=}0\ \ (P_{2\ell}(\tau){\in}M_{2\ell}).
\end{eqnarray}
When written using the usual derivative 
$\frac{d}{d\tau}$, (\ref{DE}) is a standard DE of order $n$ with coefficients in $M$.\ Consequently, the space of solutions is an $n$-dimensional linear space.\
Solutions may develop singularities at the zeros of the leading coefficient $P_{2k}(\tau)$ but are holomorphic elsewhere.

\medskip
Where is all of this leading? 
We have purposely avoided going into the theory of vector-valued modular forms, which  is actually intimately related to MLDEs.\ But by combining these two strands,
the following result can be proved (cf.\  \cite{M2}, \cite{MM}){:}
\begin{eqnarray}\label{chV}
\mbox{There is an MLDE whose solution space is exactly $\frak{ch}_V$.} 
\end{eqnarray}
Given Zhu's theory as we have outlined it, no further VOA input is needed for this result. It is  purely a matter of vector-valued modular form theory and
MLDEs.\ And in fact the theory says more than just (\ref{chV}), for in some important cases (discussed further below) it says that the MLDE can be chosen to take a specific form.

\medskip
The reason we use MLDE's rather than ordinary DEs is that there is an evident natural action of $\Gamma$ on the solution space.\ Indeed, if
$u(\tau)$ is a solution then $\gamma{\in}\Gamma$ acts via the stroke operator $u|_0\gamma(\tau)$.\ This is obvious if the solution space is $\frak{ch}_V$
for some $V$, indeed the action is just the one we described before that was found by Zhu, but it is true in general thanks to (\ref{intertwine}).\ Let's denote this representation of 
$\Gamma$ by $\rho$.\ It is in fact the \emph{monodromy representation} and one of the main features of the MLDE.\ Thus 
\begin{eqnarray*}
\mbox{\emph{the monodromy $\rho$ of an MLDE is a representation of the modular group}.}
\end{eqnarray*}

$\frak{ch}_V$ and its associated MLDE are beautiful invariants of the VOA $V$, and may be compared to the usual $S$- and $T$-matrices of RCFT, or the
modular tensor category that is $V$-Mod.\ $\frak{ch}_V$ neatly packages $S$- and $T$-matrices coming  from the monodromy representation, as well as the $q$-characters of irreducible $V$-modules that span the solution space of the MLDE.\ 
This circumstance permits us to combine techniques from the theory of modular forms and the theory of DEs in order to study $\frak{ch}_V$.

\medskip
The Problem stated at the beginning of this Section is part of the program  
\begin{eqnarray*}
\mbox{\emph{Classify rational VOAs according to $\frak{ch}_V$ and their associated MLDEs}}.
\end{eqnarray*}

In fact this effort was initiated in 1987  by the physicists Mathur, Mukhi and Sen \cite{MMS}.\ They obtained a partial result in the $2$-dimensional case
that was recently completed in \cite{MNS}.\ These papers deal with the case of \emph{monic} MLDEs of order $2$.\ Monic here
means that the leading coefficient $P_{2k}(\tau)$ in (\ref{DE}) is equal to $1$, and in particular $k{=}0$.\ Thus the DE in question is the simplest possible
order 2 MLDE, which is 
\begin{eqnarray}\label{MLDE2}
\left(D_0^2{+}\kappa E_4(\tau)\right)u{=}0\ \ (\kappa{\in}\mathbf{C})
\end{eqnarray}
where $E_4$ is the usual normalized Eisenstein series of weight $4$.\ (This comes about after perusal of (\ref{DE}) just because $M_2{=}0$ and
$M_4{=}\mathbf{C}E_4$.)

\medskip
Here are the known rational VOAs with the property that $\dim\frak{ch}_V{=}2$ and which have \emph{irreducible monodromy}.\ All but one  are affine algebras of level $1$, the other is the Yang-Lee minimal model{:}
\begin{eqnarray*}
L(A_1, 1), L(A_2, 1), L(D_4, 1), L(E_6, 1), L(E_7, 1), L(G_2, 1), L(F_4, 1), Vir_{c_{2, 5}}.
\end{eqnarray*}
These rational VOAs are almost characterized by this property in \cite{MNS}, Theorem 1.\ But it is convenient to distinguish two cases:
\begin{eqnarray*}
&&(a)\ \mbox{$V$ has exactly two irreducible modules.}\\
&&(b)\ \mbox{$V$ has \emph{more} than two irreducible modules, but $\dim\frak{ch}_V{=}2$.}
\end{eqnarray*}

Now we can state the \emph{$2$-dimensional Mathur-Mukhi-Sen Theorem}{:}
\begin{thm}\label{thm4.1}(\cite{MMS}, \cite{MNS} Theorem 2)
Suppose that $V$ is a strongly regular\footnote{Here we need the full force of strong regularity.\ That's because this condition is
assumed in some of the Theorems in the literature needed to prove  Theorem \ref{thm4.1}.} VOA satisfying (a).\ Assume that the  associated MLDE has the form (\ref{MLDE2}) and that
it has \emph{irreducible} monodromy.\ Then $V$ is isomorphic to one of the following{:}
\begin{eqnarray*}
L(A_1, 1), L(E_7, 1), L(G_2, 1), L(F_4, 1), Vir_{c_{2, 5}}.
\end{eqnarray*}
$\hfill\Box$
\end{thm}

The gist of Problem 4 is to prove a $3$-dimensional analog of Theorem \ref{thm4.1}, preferably assuming only $\dim\frak{ch}_V{=}3$ but perhaps using an 
analog of (a).\ In any case the MLDE looks like
\begin{eqnarray}\label{hyperg}
\left(D_0^3{+}\kappa E_4(\tau)D_0{+}\lambda E_6(\tau)\right)u{=}0 \ \ \ (\kappa, \lambda{\in}\mathbf{C}).
\end{eqnarray}
Typically the monodromy of such an MLDE will be irreducible, but that will \emph{not} always be the case.\ An important point that will certainly figure in the solution of
Problem 4 is that (\ref{hyperg}) is actually a disguised version of a \emph{hypergeometric equation} solved by certain hypergeometric functions ${_3}F_2$ (\cite{FM}, \cite{FM1}).

\medskip
The list of rational VOAs with $\dim\frak{ch}_V{=}3$ is bewilderingly diverse, and includes the following{:}
\begin{eqnarray}\label{VOAlist}
&&L(A_3, 1), L(A_4, 1),  L(C_2, 1), L(A_2, 2), L(E_8, 2),\notag\\
&& L(D_{\ell}, 1)\ (\ell{\geq}5),\ \ L(B_{\ell}, 1)\  (\ell{\geq}3),\notag\\
&& Vir_{c_{3, 4}}, Vir_{c_{2, 7}}, Vir_{c_{2, 5}}{\times}Vir_{c_{2, 5}}\\
&& V_{\sqrt{2}E_8},\ V_{\Lambda},\ VB^{\natural}.\notag\\
&& V^+_{\sqrt{2}E_8},\ V^+_{\Lambda}.\notag
\end{eqnarray}
Not all of these examples have irreducible monodromy.\ Among the last five examples are lattice theories $V_L$ with $L$ either a rescaled $E_8$ root lattice or the Barnes-Wall lattice 
$\Lambda$ of rank $16$ together with their $\mathbf{Z}_2$-orbifolds.\ The other one  is the Baby Monster VOA $VB^{\natural}$ \cite{H}.\ 

\medskip
I compiled this list from  various interesting papers  done in support of Problem 4 (and related problems) by Arike, Kaneko, Nagatomo and Sakai, including  characterizations of some of the VOAs \cite{AKNS}, \cite{AN}, \cite{ANS}.\ However we are still far from a complete solution.\

\medskip
Although Problem 4 is an obvious extension of  Theorem \ref{thm4.1}, fresh ideas are needed for a complete solution of Problem 4.\ This is because the list (\ref{VOAlist}) contains \emph{infinitely many} different VOAs, and  in particular there is no bound on the possible central charge $c$.\ In the proof of 
Theorem \ref{thm4.1} in \cite{MNS} one first shows that the assumptions of the Theorem imply that there are only finitely many possible values of $c$, then a detailed analysis
of each possibility yields the final answer.\ So this approach must be modified.\  A modest start  is made in \cite{MNS1}.

\medskip
I would have liked to discuss some additional questions about the MLDEs attached to rational VOAs, a subject which I find fascinating.\ For example \textit{are these MLDEs  necessarily Fuchsian on the $3$-punctured sphere?}\ However an adequate discussion would require more space than is available here.

\section{Pierce bundles of local vertex rings}

\noindent
$\mathbf{Problem\ 5}.$   Characterize \emph{exchange} vertex rings.

\medskip
In order to explain the meaning of this Problem we will  first make an incursion into the theory of  \emph{Pierce bundles} of a commutative ring.\ This may seem like something of detour to the reader, but it is vital to explain and motivate Problem 5.\
Pierce's original paper \cite{P} is, of course, a good reference for his theory, and the intriguing text of Johnson \cite{J} also deals with this set of ideas in a broad
framework.\ The extension to vertex rings was first given in Part II of \cite{M3}.

\medskip
The \emph{structure sheaf} of a unital, commutative ring $k$ is a fundamental geometric object \cite{EGA} that  has now infiltrated  into graduate algebra texts (e.g., \cite{DF}, Chapter 15).\
The same cannot be said of the Pierce sheaf of $k$, though they are closely related.\ 

\medskip
Let $X{:=}Spec\ k$ be the prime ideal spectrum of $k$ equipped with the Zariski topology.\
One considers the disjoint union $E{:=}\bigcup_P k_P$ of the localizations $k_P$ of $k$ at a prime ideal $P{\subseteq}k$.\ Once $E$ is topologized, we get a bundle of commutative rings $\pi{:}E\rightarrow X$ and the fiber $\pi^{-1}(P)$ over $P$ is just $k_P$.\ Such a bundle gives rise to the structure sheaf $\mathcal{O}$ of $k$; it is a sheaf of rings over $X$ which arises from the  \emph{local sections} of $\pi$ over the open sets of $X$.\ There is a categorical equivalence between such bundles and sheaves over $X$,
so that they carry the same data.

\medskip
In the Pierce construction, one begins with the same $k$ but a very different $X$.\ Namely associated to $k$ is the set of all
\emph{idempotents} $e{\in}k$.\ This set carries the structure of a \emph{Boolean algebra} and it may also be regarded as a second commutative ring in which
multiplication of idempotents $e, f$ is just their product $ef$ in $k$, whereas addition is defined by $e\oplus f{:=}e{+}f{-}2ef$\ ($e{+}f$ is addition in $k$).\ 
Denote this ring by $B{=}B(k)$.\ It is a \emph{Boolean ring} in the sense that every element is idempotent.\ For further details about this construction, cf. \cite{P} and \cite{Jac}, Chapter 8.\

\medskip
For the purposes of the Pierce construction we take $X{:=}Spec\ B(k)$.\ This is an example of a \emph{Boolean space}, namely the Zariski topology on $X$ is both \emph{Hausdorff} and
\emph{totally disconnected} (connected components are single points) and the topology has a \emph{basis of clopen sets}.\ 

\medskip
The \emph{Pierce bundle}
$\pi{:}E\rightarrow X$ of $k$ is defined as follows{:}\ as in the case of the structure sheaf $\mathcal{O}$, $E$ is defined as the disjoint union of what will eventually be the stalks,
and if $P{\in}X$ is a prime ideal of $B$ then $\pi^{-1}(P){:=}k/\overline{P}$ where $\overline{P}{:=}\cup_{e\in P} ek$.\ ($\overline{P}$ really \emph{is} an ideal in $k$, though the union may suggest otherwise).

\medskip
Pierce showed \cite{P} that $k$ can be recovered as the \emph{global sections} of $\pi$.\ Furthermore, the stalks $\pi^{-1}(P)$ are \emph{indecomposable rings}.\
This sets up an equivalence between the category
$Ring$ of commutative rings and the category of \emph{reduced bundles}, which are \`{e}tale bundles of indecomposable rings over a Boolean base space.

\medskip
Pierce's work gave impetus to a cottage industry focused on the question of what commutative rings can arise as global sections of suitably conditioned  bundles of rings?\ 
This is sometimes referred to as representation theory of rings.

\medskip
For example,  Pierce obtained \cite{P} a beautiful  characterization of those commutative rings $k$ having the additional property that the stalks of the Pierce bundle are not just indecomposable, but in fact
\emph{simple}.\ Of course, a simple commutative ring is a \emph{field}, so the question amounts to this:\  which commutative rings have Pierce bundles whose stalks are fields?
Remarkably, these are precisely the commutative \emph{von Neumann regular rings}.\ This class of commutative rings may be characterized in several ways, one of which is that every
principal ideal is generated by an idempotent.\ For the general theory of (not necessarily commutative) von Neumann regular rings, see \cite{G}.

\medskip
As a second illustration, and one which almost  brings us to the meaning of Problem 5, one may pose the following question{:}\  what commutative rings have Pierce bundles whose stalks are
\emph{local rings}?\ This is meaningful inasmuch as a local ring is necessarily indecomposable.\ On the other hand, local rings include fields, so that after Pierce's theorem  the class of rings that we are after here includes all commutative von Neumann regular rings.\ The question was answered by Monk \cite{Monk} following work of Warfield \cite{W}.\ See also \cite{J}, Chapter $V$.\
Monk showed that \emph{exchange rings} are precisely the rings whose Pierce bundles have stalks which are local rings.\ Here, a commutative exchange ring $k$
is defined by the following property{:}
\begin{eqnarray}\label{exchcond}
\mbox{every element of $k$ is the sum of an idempotent and a unit.}
\end{eqnarray}

\medskip
At last we turn to vertex rings.\ In \cite{M3}, Part II, I showed that Pierce's theory as we have sketched it out, extends naturally to the category $Ver$ of 
vertex rings in place of $Ring$.\ In particular, every vertex ring $V$  has a Pierce bundle, which is an \`{e}tale bundle of indecomposable vertex rings
$\pi{:}E\rightarrow X$ over a Boolean base space $X$.\ $V$ may be recovered as the vertex ring of global sections of $\pi$.\ An important point is the origin of the base space $X$.\ Indeed,  define the \emph{center} $C(V)$ of $V$ to consist of states $v{\in}V$ whose vertex operator $Y(v, z){=}v({-}1)$ is \emph{constant}.\ Then $C(V)$ is a unital commutative ring with respect to the ${-}1^{th}$ product in $V$ and a (highly degenerate) vertex subring of $V$.\ Then we take
\begin{eqnarray*}
X{:=} B(C(V)).
\end{eqnarray*}

I also obtained (\cite{M3}, Theorem 11.1) the analog of Pierce's characterization of von Neumann regular rings.\ Surprisingly, the result is a more-or-less \emph{verbatim} restatement of Pierce's theorem, but with `vertex ring' in place of `commutative ring'.\ 

\medskip
To be explicit, let $V$ be a vertex ring.\ The 2-sided principal ideal generated by an element $a{\in}V$ is the intersection of all 2-sided ideals of $V$ that contain $a$, and an idempotent of $V$ is just an idempotent of $C(V)$ regarded as a commutative ring.\ Then we call $V$ a \emph{von Neumann regular vertex ring} if every 2-sided principal ideal of $V$ is of the form $e({-}1)V$ for some idempotent $e{\in}C(V)$.\ Let $E\rightarrow X$ be the Pierce bundle of $V$.\ Then 

\begin{eqnarray}\label{VNVRings}
&&\mbox{$V$ is a von Neumann regular vertex ring if, and only if, the stalks}\\
&&\mbox{of the Pierce bundle are \emph{simple} vertex rings}.\notag
\end{eqnarray}

\medskip
Problem 5 asks for the vertex ring analog of Monk's theorem about exchange  rings.\ 

\medskip
To be clear, 
define a vertex ring $V$ to be a \emph{local vertex ring} if $V$ has a \emph{unique maximal ideal} $J$ ($J$ is maximal if, and only if, $V/J$ is a simple vertex ring).\
Such vertex rings are very familiar.\ If $V$ is a VOA over $\mathbf{C}$ of CFT-type,  then its conformal grading takes the shape
\begin{eqnarray*}
V{=}\mathbf{C}\mathbf{1}{\oplus}V_1{\oplus}V_2\hdots
\end{eqnarray*}
Every such $V$ that is \emph{self-dual} is a local VOA, where $J$ is the radical of any nonzero invariant bilinear form on $V$.

\medskip
Cue Problem 5{:}\ 
\begin{eqnarray*}
&&\mbox{\textit{which vertex rings $V$ have the property that their Pierce bundles}}\\
&&\mbox{\textit{have stalks all of which are local vertex rings?}} 
\end{eqnarray*}
This simultaneously generalizes (\ref{VNVRings}) and Monk's theorem.\ We may call such a vertex ring an \emph{exchange vertex ring}.\ So how should we generalize (\ref{exchcond}) so that it also applies to all vertex rings and not just commutative rings?

\section{Genus~$2$ Monstrous Moonshine}\label{S6}

\noindent
$\mathbf{Problem\ 5}$.\ Construct the genus $2$  Moonshine Module $V^{\natural}$.

\medskip
We have already referred to the influence of Zhu's paper \cite{Z}, in which he treated the arithmetic properties of the characters of modules of a rational VOA.\ In fact Zhu did much more than this.\
He defined $n$-point functions, and implicitly reduced their study to that of $1$-point functions by establishing a recursive formula that they satisfy.\ These functions all reside on a complex torus, i.e., a compact  Riemann surface of genus $g{=}1$.\ On the other hand, and in marked contrast to the physics literature where the use of `higher loop' calculations is routine, there have been relatively few applications of $n$-point functions on a higher genus Riemann surface in the mathematical literature on VOAs.\ 

\medskip
Following his initial paper \cite{T}, Michael Tuite  and I set out to do something about this \cite{MT1}-\cite{MT6}.\ Because the `easiest' higher loop functions to study
ought to be $0$-point functions on a $g{=}2$ Riemann surface, i.e., the \emph{genus $2$ characters of a VOA $V$}, we introduced some functions that we  expect will serve as
the desired genus $2$ characters.\ With the idea of modular-invariance in mind, we anticipated  that these functions would have some $g{=}2$ modular-invariance properties, and for rational $V$ would in fact be Siegel modular forms on  a congruence subgroup of $Sp_4(\mathbf{Z})$.

\medskip
Problem 5 asks for the determination of the genus $2$ character of the Moonshine Module $V{=}V^{\natural}$.\ A precise Conjecture is stated at the end of this Section.\
Although Tuite and I managed to understand the case 
when $V$ is a lattice theory $V_L$ (more on this below) the case of $V^{\natural}$ eluded us.\ I've always found this circumstance to be rather disappointing, and would be delighted to see the solution!\
In the rest of this Section I will give an account of the genus $2$ theory for $V_L$.\ In effect,  problem 5 asks a reproduction of the calculations that follow for related VOAs such as
$V_L^+$ and $\mathbf{Z}_2$-orbifolds of $V_L$.

\medskip
A distinctive feature of the higher genus theory for VOAs $V$  is that it  emphasizes, and depends on, the geometric theory of Riemann surfaces of all genera, not just tori.\ In keeping with Zhu's perspective, we define the higher genus $n$-point functions of $V$ in a sort of recursive manner.\ Thus we assume we have at our disposal the full gamut of  $n$-point functions for $V$ at $g{=}1$, and  then define the higher genus $n$-point functions in terms of them.

\medskip
When $g{=}2$, this leads to the circumstance that we obtain not just one, but  two, different definitions of the genus $2$ character of $V$.\ This is because
there are two rather different ways to construct a $g{=}2$ compact Riemann surface by doing some plumbing (sewing) with $g{=}1$ surfaces (complex tori).\ In the first approach (the so-called $\epsilon$-formalism) we sew together a pair of complex tori $X_1, X_2$ by excising a parameterized disk of radius $\epsilon$ from each $X_i$ and identifying the resulting boundaries.\  Here we are plumbing $X_1$ and $X_2$ by connecting them, if you will, using a cylinder running between the punctures.\ The resulting surface is often denoted by $X_1\#X_2$.\ In the second approach (the $\rho$-formalism)
we start out with a single complex torus $X$, excise a disjoint pair of disks, and attach a handle running from one  to the other.

\medskip
Let $V$ be a VOA.\ For $v{\in}V$ the vertex operator for $v$ is denoted by 
\begin{eqnarray*}
Y(v, z){:=}\sum_{n{\in}\mathbf{Z}} v(n)z^{-n-1}
\end{eqnarray*}
The \emph{zero mode} of a homogeneous state $v{\in}V_k$  is $o(v){:=}v(k{-}1)$.\ This operator preserves all homogeneous spaces $V_{\ell}$.\ For general
$v{\in}V$ written as a sum of homogeneous states we define $o(v)$ by linear extension.\ The $1$-point functions at genus $1$ for $V$ are defined as follows{:}
\begin{eqnarray*}
&&Z^{(1)}(v, \tau){:=}\Tr_V Y(q^{L(0)}v,q)q^{L(0)-c/24}\\
&&{=} \Tr_V q^{L(0){-}c/24}o(v){=}q^{-c/24}\sum_{\ell} \Tr_{V_{\ell}}o(v)q^{\ell}
\end{eqnarray*}
where, as usual, we have set $q{:=}q_{\tau}{:=}e^{2\pi i\tau}$ for $\tau$ in the complex upper half-plane $\mathbf{H}$.\ The $2$-point functions at $g{=}1$ may be defined as follows{:}
\begin{eqnarray*}
&&Z^{(1)}((v_1, q_1), (v_2, q_2)){:=}  \Tr_VY(q_1^{L(0)}v_1, q_1) Y(q_2^{L(0)}v_2, q_2)q^{L(0){-}c/24},
\end{eqnarray*}
where $q_j{:=}q_{\tau_j}$.\ For additional background, cf.\ \cite{Z}, \cite{MT1}, \cite{MT4}, \cite{MT5}.

\medskip
We come to the genus 2 character in the $\epsilon$-formalism, which may be defined for nice enough $V$ as follows{:}
\begin{eqnarray}\label{g=2Z}
Z^{(2)}_{V, \epsilon}(\tau_1, \tau_2, \epsilon){=}\sum_{k\geq 0} \epsilon^k\sum_{u\in V_{[k]}} Z^{(1)}_V(u, \tau_1)Z^{(1)}_V(\overline{u}, \tau_2).
\end{eqnarray}
We must explain the notation.\ We are assuming that $V{=}\mathbf{C}\mathbf{1}{\oplus}V_1{\oplus}\hdots$ is a \emph{self-dual} VOA of CFT-type, so that $V$ comes equipped with a
symmetric, nondegenerate,  invariant, bilinear form.\ Associated to $V$ is an \emph{isomorphic} copy $V[\ ]$ that arises from the \emph{same} underlying linear space
and a change-of-variables vertex operator (loc.\ cit.)
\begin{eqnarray*}
Y[v, z]{:=}Y(q_z^{L(0)}v, q_z{-}1),
\end{eqnarray*}
 and $V[\ ]{:=}\mathbf{C}\mathbf{1}{\oplus}V_{[1]}\oplus \hdots$ is the (isomorphic) conformal grading.\ Let $b$ be the normalized invariant bilinear form
 on $V[\ ]$.\ For each $k$, $b$ sets-up an identification of  $V_{[k]}$ with its dual space, and $\overline{u}{\in}V_{[k]}$ is the
\emph{metric dual} of $u$ with respect to $b$, i.e., $\bar{u}{:=}\sum_i (b, u_i)u_i$ for a basis $\{u_i\}$  of $V_{[k]}$ satisfying $b(u_i, u_j){=}\delta_{ij}$.\
This change of variables corresponds to a pivot from the original VOA $(V, Y, \mathbf{1})$ on a punctured sphere to an isomorphic copy on 
$(V[\ ], Y[\ ], \mathbf{1})$ on a cylinder.\ Then (\ref{g=2Z}) relates $1$-point functions on the cylinder to the plumbed $g{=}2$ Riemann surface $X_1\#X_2$ with parameter
$\epsilon$.

\medskip
We think of $\tau_i$ as the \emph{modulus} of the elliptic curve $X_i$, i.e., the point in $\mathbf{H}$ corresponding to $X_i$.\ Then there is a natural complex domain denoted (somewhat confusingly) by
$D^{\epsilon}$ consisting of the triples $(\tau_1, \tau_2, \epsilon)$ for which the sewing procedure producing $X_1\#X_2$ is defined, and there is an evident  diagram 
\begin{eqnarray*}
\xymatrix{    D^{\epsilon} \ar[rr]^{\Omega^{(2)}}\ar[drr]_{Z_{V, \epsilon}^{(2)}}& &\mathbf{H}^2&\\
&&
  \mathbf{C}&&
}
\end{eqnarray*}
where $\Omega^{(2)}$ is the \emph{period matrix} of $X_1\#X_2$, which lies in the \emph{genus $2$ Siegel upper half space} $\mathbf{H}^{(2)}$.\ (For background, cf.\ \cite{F}).\ It is known \cite{MT2}, \cite{Y} that $\Omega^{(2)}$, which describes  the period matrix in terms of sewing variables,  is a holomorphic map.\ Furthermore $Z^{(2)}$ is holomorphic \cite{MT4}.\ We may then ask{:} 
\begin{eqnarray}\label{Q}
\mbox{does $Z^{(2)}$ factor through $\Omega^{(2)}$?}
\end{eqnarray}

\medskip
There is a parallel story in the $\rho$-formalism, where the genus 2 character is defined by
\begin{eqnarray}\label{g=2R}
Z^{(2)}_{V, \rho}(\tau, \rho){=}\sum_{k\geq 0} \rho^k\sum_{u\in V_{[k]}}Z^{(1)}_V(\overline{u}, u, w,  \tau).
\end{eqnarray}
Recall that in this formalism we are attaching a handle to an elliptic curve $X$ of modulus $\tau$.\ This procedure defines the complex domain $D^{\rho}$ consisting of the
data $(\tau, \rho, w)$  needed to implement the plumbing \cite{MT6}.\ And there is a second diagram of holomorphic maps
\begin{eqnarray*}
\xymatrix{    D^{\rho} \ar[rr]^{\Omega^{(2)}}\ar[drr]_{Z_{V, \rho}^{(2)}}& &\mathbf{H}^2&\\
&&
  \mathbf{C}&&
}
\end{eqnarray*}
Of course we still want to know the answer to (\ref{Q}) in the $\rho$-formalism.

\medskip
Let us now turn to a description of the genus $2$ characters for lattice theories $V{=}V_L$.\ If we want to compare these functions with, say, Siegel modular forms, there is an evident difficulty, namely the \emph{mismatch of variables}.\ For example, the genus 2 theta-function for $L$ is defined in terms of a symmetric matrix $\Omega$ with positive-definite real part, 
that is $\Omega{\in}\mathbf{H}^{(2)}$, which we may think of as the period matrix of a $g{=}2$ Riemann surface:
\begin{eqnarray*}
\Theta_L^{(2)}(\Omega){:=}\sum_{\alpha, \beta{\in}L} exp \left\{  \pi i \left(  (\alpha, \alpha)\Omega_{11}{+}2(\alpha, \beta)\Omega_{12}{+}(\beta, \beta)\Omega_{22}\right) \right\}.
\end{eqnarray*}
This theta-function has the entries $\Omega_{ij}$ of $\Omega$ as variables, and these are vastly different to the variables arising from the sewing procedures.

\medskip
Now the answers to (\ref{Q}) and its $\rho$-analog is  \textit{no, but morally yes}.\ To explain what this means, recall that the genus $1$ character for
a rank $n$ lattice theory $V_L$ is
\begin{eqnarray*}
Z_{V_L}^{(1)}(\tau){=}\frac{\theta_L(\tau)}{\eta(\tau)^n}.
\end{eqnarray*}
This arises from the containment of a rank $n$ Heisenberg subVOA $M^n{\subseteq}V_L$, indeed we may rewrite the last display as 
\begin{eqnarray*}
\frac{Z_{V_L}^{(1)}(\tau)}{Z_{M^n}^{(1)}(\tau)}{=}\theta_L(\tau).
\end{eqnarray*}
Surprisingly,  precisely the same formula holds at genus $2$ in both formalisms \cite{MT2}, \cite{MT6}.\ To be somewhat misleading, we have
\begin{eqnarray*}
\frac{Z_{V_L, \epsilon}^{(2)}(\tau_1, \tau_2, \epsilon)}{Z_{M^n, \epsilon}^{(2)}(\tau_1, \tau_2, \epsilon)}{=}\theta_L^{(2)}(\Omega){=}\frac{Z_{V_L, \rho}^{(2)}(\tau, \rho, w)}{Z_{M^n, \rho}^{(2)}(\tau, \rho, w)}.
\end{eqnarray*}
(It is misleading because $\Omega$ really depends on the moduli, depending on which formalism we are in.)\ A more accurate statement is that the following diagrams commute{:}
\begin{eqnarray*}
\xymatrix{    D^{\epsilon} \ar[rr]^{\Omega^{(2)}}\ar[drr]_{Z_{V_L, \epsilon}^{(2)}/Z^{(2)}_{M^n,\epsilon}}& &\mathbf{H}^2\ar[d]^{\Theta_L^{(2)}}&\\
&&
  \mathbf{C}&&
}
\xymatrix{    D^{\rho} \ar[rr]^{\Omega^{(2)}}\ar[drr]_{Z_{V_L, \rho}^{(2)}/Z^{(2)}_{M^n, \rho}}& &\mathbf{H}^2\ar[d]^{\Theta_L^{(2)}}&\\
&&
  \mathbf{C}&&
}
\end{eqnarray*}

\medskip
We can now prescribe exactly what is needed in Problem 5.\ Regarding any of the VOAs $V{=}V_L^+$ or its $\mathbf{Z}_2$-orbifold, what is required  are analogs of the last display.\ The conjecture is that
there is a Siegel modular form $F$ (depending on $V$) such that the following commute{:}
\begin{eqnarray*}
\xymatrix{    D^{\epsilon} \ar[rr]^{\Omega^{(2)}}\ar[drr]_{Z_{V, \epsilon}^{(2)}/Z^{(2)}_{M^n,\epsilon}}& &\mathbf{H}^2\ar[d]^F&\\
&&
  \mathbf{C}&&
}
\xymatrix{    D^{\rho} \ar[rr]^{\Omega^{(2)}}\ar[drr]_{Z_{V, \rho}^{(2)}/Z^{(2)}_{M^n, \rho}}& &\mathbf{H}^2\ar[d]^F&\\
&&
  \mathbf{C}&&
}
\end{eqnarray*}

In case $L{=}\Lambda$ is the \emph{Leech lattice} and $V{=}V^{\natural}$ is the $\mathbf{Z}_2$-orbifold, alias the Moonshine Module, $F$ will be long sought
 {genus 2 Moonshine character}.\ 
 
 \medskip
 None of this involves the Monster simple group \emph{per se}, however if the existence of $F$ can be established, one expects that similar
 results can be proved for $1$-point functions at $g{=}2$, and also including twisted sectors and group elements along the lines of \cite{DLM}.

\bibliographystyle{amsplain}

\end{document}